\documentclass[12pt]{article}
\usepackage[latin1]{inputenc}
\usepackage[T1]{fontenc}
\usepackage{amsmath,amssymb,amstext}
\usepackage{theorem}
\usepackage{a4wide}
\usepackage{multicol}
\usepackage[english]{babel}
\usepackage{enumerate}
\usepackage{eufrak}

\def\HH{\mathcal{H}}
\def\R{{ \mathbb{R}}}
\newtheorem{theorem}{Theorem}

\newtheorem{lemma}[theorem]{Lemma}

\newtheorem{proposition}[theorem]{Proposition}
\newtheorem{remark}[theorem]{Remark}

\newenvironment{proof}[1][Proof]{\textbf{#1.} }{\ \rule{0.5em}{0.5em}}
\def\cal{\mathcal}

\renewcommand{\geq}{\geqslant}
\def\leq{\leqslant}

\def\HH{{\cal H}}

\def\cal{\mathcal}
\def\1{{\mathbf{1}}}

\def\Un{{\bf 1}}

\def\sk{{\mathbb{D}}}

\def\oT{{[0{,}T]}}

\def\ffi{{\varphi}}

\def\1{{\mathbf{1}}}
\def\0.5{{\frac{1}{2}}}

\renewcommand{\thefootnote}{\fnsymbol{footnote}}

\begin{document}

\renewcommand{\thefootnote}{\arabic{footnote}}

\begin{center}
{\Large \textbf{Least squares estimator of fractional Ornstein Uhlenbeck processes  with periodic mean }} \\[0pt]
~\\[0pt]
Salwa Bajja\footnote{%
 National School of Applied Sciences - Marrakesh, Cadi Ayyad
University, Marrakesh, Morocco. Email: \texttt{salwa.bajja@gmail.com }}, Khalifa Es-Sebaiy\footnote{%
National School of Applied Sciences - Marrakesh, Cadi Ayyad
University, Marrakesh, Morocco. Email: \texttt{k.essebaiy@uca.ma}}
and Lauri Viitasaari
\footnote{%
Department of Mathematics and System Analysis, Aalto University School of
Science, P.O. Box 11100, FIN-00076 Aalto, Finland. E-mail:\texttt{lauri.viitasaari@aalto.fi}}%
\\[0pt]
\textit{Cadi Ayyad University and Aalto University}\\[0pt]
~\\[0pt]
\end{center}
\begin{abstract}
\medskip
 We first study the drift parameter estimation of the fractional  Ornstein-Uhlenbeck process (fOU) with
periodic mean for every $\frac{1}{2}<H<1$. More precisely, we extend
the consistency proved in \cite{DFW} for $\frac{1}{2}<H<\frac{3}{4}$
to the strong consistency for  any $\frac{1}{2}<H<1$ on the one
hand, and on the other, we also discuss the asymptotic normality
given in \cite{DFW}. In the second main part of the paper, we  study
the strong consistency and the asymptotic normality of the fOU of
the second kind with periodic mean for any $\frac{1}{2}<H<1$.
\end{abstract}
\noindent {{\bf Keywords:}} Fractional Ornstein-Uhlenbeck processes;
Least squares estimator; Malliavin calculus.
\section{Introduction}
Consider the fractional Ornstein-Uhlenbeck process (fOU)
$X=\left\{X_{t},t\geq0\right\}$ given by the following linear
stochastic differential equation
\begin{eqnarray}
  dX_{t} &=& -\alpha X_{t}dt+dB_{t}^{H},\ \ X_{0}=0,\label{fOU-equ}
\end{eqnarray}
where $\alpha$ is an unknown parameter, and $B^H = \left\{B^H_t, t
\geq 0\right\}$ is a fractional Brownian motion (fBm) with Hurst
parameter $H \in ( 0
, 1)$.\\
The drift parameter estimation problem  for the fOU $X$ observed in
continuous time and discrete time has been studied by using several
approaches (see \cite{KL,HN, HS, BI, EEV, EN}). In a general case
when the process $X$ is driven by a Gaussian process, \cite{EEO}
studied the non-ergodic case corresponding to $\alpha<0$. They
provided sufficient conditions, based on the properties of the
driving Gaussian process, to ensure that least squares
estimators-type of $\alpha$ are strongly consistent and
asymptotically Cauchy. On the other hand, using  Malliavin-calculus
advances (see \cite{NP}), \cite{EV} provided new techniques to
statistical inference for stochastic differential equations related
to stationary Gaussian processes, and they used their result to
study drift parameter estimation problems for some stochastic
differential equations driven by fractional Brownian motion with
fixed-time-step observations (in particular for the fOU $X$ given in
(\ref{fOU-equ}) with $\alpha>0$). Similarly, in \cite{sot-vii} the
authors studied an ergodicity estimator for the parameter $\alpha$
in \eqref{fOU-equ}, where the fractional Brownian motion is replaced
with a general Gaussian
 process having stationary increments.

 Recently,  \cite{DFW} studied a drift parameter estimation problem for  the above equation (\ref{fOU-equ})
 with slight modifications on the drift. More precisely, they considered the
following fractional Ornstein Uhlenbeck process with periodic mean
function
\begin{eqnarray}
  dX_{t} &=& \left(\sum_{i=1}^{p}\mu_{i}\varphi_{i}(t)-\alpha X_{t}\right)dt+dB^{H}_{t},\quad X_0=0
  \label{FOU with PF}
\end{eqnarray}
where $B^H$ is a fBm   with Hurst parameter $ \frac12<H<1$, the
functions $\varphi_{i}, i=1,\ldots,p$ are bounded by a constant
$C>0$ and periodic with the same period $\nu>0$, and the real
numbers $\mu_{i}, i=1,\ldots,p$ together with  $\alpha>0$ are
considered unknown parameters. The motivation comes from the fact that such equation can be used to model time series which are a combination of a stationary process and periodicities. In \cite{DFW} the authors proposed the least squares
estimator (LSE) to estimate
$\theta:=(\mu_{1},\ldots,\mu_{p},\alpha)^{\top}$ based on the
continuous-time observations
 $\{X_t,0\leq t\leq n\nu\}$ as $n\rightarrow\infty$.
For the sake of simplicity, we assume that the functions
$\varphi_{i}, i = 1,\ldots, p$
 are orthonormal in $L^{2}([0, \nu], \nu^{-1}dt)$, i.e.
$\int_0^{\nu}\varphi_{i}(t)\varphi_{j}(t)\nu^{-1}dt = \delta_{ij}$.
We also choose $\nu=1$.

Let us consider the LSE $\widehat{\theta}_{n}$ of $\theta$ given in
\cite{DFW}  by
\begin{eqnarray}
  \widehat{\theta}_{n} &:=&Q_{n}^{-1}P_{n} \label{LSE FOU with PF}
\end{eqnarray}
where
\begin{eqnarray*}
  P_{n} :=\left(\int_{0}^{n}\varphi_{1}(t)dX_{t},\ldots,\int_{0}^{n}\varphi_{p}(t)
  dX_{t},-\int_{0}^{n}X_{t}dX_{t}\right)^{\top}, \quad Q_{n} &=& \left(
              \begin{array}{cc}
                G_{n} & -a_{n} \\
                -a^{\top}_{n} & b_{n} \\
              \end{array}
            \right)
\end{eqnarray*}
with
\begin{eqnarray*}
  G_{n} &:=& \left(\int_{0}^{n} \varphi_{i}(t)\varphi_{j}(t)dt  \right)_{1\leq i,j\leq p};
\end{eqnarray*}
\[a_{n}^{\top}:=\left(\int_{0}^{n}\varphi_{1}(t)X_{t}dt,\ldots,\int_{0}^{n}\varphi_{p}(t)X_{t}dt\right);
\quad b_{n}:=\int_{0}^{n}X_{t}^{2}dt.\] Let us describe what is
known about the asymptotic behavior of $\widehat{\theta}_n$: if
$\frac12<H<\frac34$, then
\begin{itemize}
\item as $n\rightarrow\infty$,
 \begin{equation}\label{consistency FOU
with PF} \widehat{\theta}_n \longrightarrow\theta,\end{equation} in
probability, see \cite[Theorem 1]{DFW};
\item
as $n\rightarrow\infty$,
\begin{equation}
n^{1-H}\big(\widehat{\theta}_n-\theta\big) \mbox{ converges in law
to a normal distribution},
\end{equation} see  \cite[Theorem 2]{DFW}.
\end{itemize}

 In the first main part of our paper  we extend the convergence in probability
(\ref{consistency FOU with PF}) proved when $\frac12<H<\frac34$ to
the almost sure convergence  for  every $\frac12<H<1$. More
precisely, we establish the strong consistency for the LSE
$\widehat{\theta}_n$ for  every $\frac12<H<1$. On the other hand, in
Theorem \ref{asymp. norm. FOU with PF} we correct the covariance
matrix of the normal limit distribution given in \cite[Theorem
1]{DFW} because the proof of \cite[Proposition 5.1]{DFW} relies on a
possibly flawed technique in line -2 page 13.

Our second main interest in this paper is to estimate the drift
parameters of the fractional Ornstein Uhlenbeck process of the
second kind with periodic mean, that is the solution of the
following equation
\begin{eqnarray}
\label{eq: SDE second kind fOU}
  dX^{(1)}_{t} &=& \left(\sum_{i=1}^{p}\mu_{i}\varphi_{i}(t)-\alpha X^{(1)}_{t}\right)dt+dY_{t}^{(1)},
   \ \ X_{0}=0
\end{eqnarray}
where $Y_{t}^{(1)}:=\int_{0}^{t}e^{-s}dB^H_{a_{s}}$ with
$a_{t}=He^{\frac{t}{H}}$ and  $B^H$ is a fBm with Hurst parameter
$\frac12<H<1$. The parameter estimation for the fOU  of the second
kind without periodicities is well studied in
 several recent papers (see    \cite{EV, AM, AV, EET}).\\
Let $ \widetilde{\theta}_{n}$ be  the LSE of $\theta$ defined by
\begin{eqnarray}
\label{LSE FOUSK PF}
\widetilde{\theta}_{n}=\widetilde{Q}_{n}^{-1}\widetilde{P}_{n}
\end{eqnarray}
where
\begin{eqnarray*}
  \widetilde{P}_{n} :=\left(\int_{0}^{n}\varphi_{1}(t)dX^{(1)}_{t},\ldots,\int_{0}^{n}
  \varphi_{p}(t)dX^{(1)}_{t},-\int_{0}^{n}X^{(1)}_{t}dX^{(1)}_{t}\right)^{\top};
   \quad \widetilde{Q}_{n} &=& \left(
              \begin{array}{cc}
                G_{n} & -\widetilde{a}_{n} \\
                -\widetilde{a}^{\top}_{n} & \widetilde{b}_{n} \\
              \end{array}
            \right)
\end{eqnarray*}
with  $G_{n}$ is given as in above, and
\[\widetilde{a}_{n}^{\top}:=\left(\int_{0}^{n}\varphi_{1}(t)X^{(1)}_{t}dt,\ldots,
\int_{0}^{n}\varphi_{p}(t)X^{(1)}_{t}dt\right); \quad
\widetilde{b}_{n}:=\int_{0}^{n}(X^{(1)}_{t})^{2}dt.\] Let us now
describe the results we establish for the asymptotic behavior of the
LSE  $\widetilde{\theta}_{n}$: if $\frac12<H<1$, then
\begin{itemize}
\item as
$n\rightarrow\infty$,
 \begin{equation*} \widetilde{\theta}_n \longrightarrow\theta,\end{equation*}
almost surely, see Theorem  \ref{consitency FOUSK with PF};
\item as
$n\rightarrow\infty$,
\begin{eqnarray*}
  \sqrt{n}(\widetilde{\theta}_{n}-\theta) &\overset{law}{\longrightarrow}&
  \mathcal{N}(0,\overline{M}^{\top}\overline{\Sigma}\ \overline{M})
\end{eqnarray*} see  Theorem \ref{asymp. norm. FOUSK with PF}.
\end{itemize}

Our article is structured as follows. In section 2, we establish the
strong consistency for the LSE $\widehat{\theta}_n$  for every
$\frac12<H<1$. Moreover, we discuss the asymptotic normality given
in  \cite{DFW}. Section 3  is devoted to  study the strong
consistency and the asymptotic normality for the LSE
$\widetilde{\theta}_n$ for any $\frac12<H<1$. Finally, some basic
elements of Malliavin calculus with respect to fBm which are helpful
for some of the arguments we use, and some of the technical results
used in various proofs are in the Appendix.

\section{ LSE for fOU  with periodic mean}
From (\ref{LSE FOU with PF}) and (\ref{FOU with PF}) we can write
(see \cite{DFW} for details)
\begin{eqnarray}
\label{decomp. of LSE fOU PF}
 \widehat{\theta}_{n}&=&\theta+ Q_{n}^{-1}R_{n}
\end{eqnarray}
with an explicit expression of the matrix $Q_{n}^{-1}$
\begin{eqnarray}
\label{eq:matrice Q}
  Q_{n}^{-1} &=&\frac{1}{n}\left(
          \begin{array}{cc}
            I_{p}+ \gamma_{n} \Lambda_{n} \Lambda_{n}^{\top} & -\gamma_{n} \Lambda_{n} \\
            -\gamma_{n} \Lambda_{n}^{\top} & \gamma_{n} \\
          \end{array}
        \right) \nonumber
\end{eqnarray}
and
\begin{eqnarray*}
\label{eq:vector R } R_{n} :=
\left(\int_{0}^{n}\varphi_{1}(t)dB_{t}^{H},\ldots,\int_{0}^{n}
\varphi_{p}(t)dB_{t}^{H},-\int_{0}^{n}X_{t}\delta
B_{t}^{H}\right)^{\top},
\end{eqnarray*}
where
\begin{eqnarray*}
  \Lambda_{n} &=& (\Lambda_{n,1},\ldots,\Lambda_{n,p})^{\top}:=\left(\frac{1}{n}\int_{0}^{n}\varphi_{1}
  (t)X_{t}dt,\ldots,\frac{1}{n}\int_{0}^{n}\varphi_{p}(t)X_{t}dt\right)^{\top}\\
\end{eqnarray*}
and
\begin{eqnarray*}
   \gamma_{n}&:=&\left(\frac{1}{n}\int_{0}^{n}X_{t}^{2}dt-\sum_{i=1}^{p}\Lambda^{2}_{n,i}\right)^{-1}.
\end{eqnarray*}
 On the other hand, it is readily checked that we have the
following explicit expression for the solution $X$ of (\ref{FOU with
PF})
\begin{eqnarray}
 X_{t} &=& h(t)+Z_{t},\quad t\geq0, \label{explicit exp. of FOU PF}
\end{eqnarray}
where
\begin{eqnarray}
  h(t) := e^{-\alpha t}\sum_{i=1}^{p}\mu_{i}\int_{0}^{t}e^{\alpha
  s}\varphi_{i}(s)ds,\quad Z_{t}&:=& e^{-\alpha t}\int_{0}^{t}e^{\alpha
  s}dB_{s}^{H}.\label{exp. of h and Z}
\end{eqnarray}
Moreover the process $Z$ is a fOU, that is solution of the following
equation
\begin{eqnarray}
  dZ_{t} &=&  -\alpha Z_{t}dt+dB^{H}_{t},\quad Z_0=0.\label{FOU}
\end{eqnarray}  The following result establishes the strong consistency
of the LSE $\widehat{\theta}_{n}$.
\begin{theorem}Assume that $\frac{1}{2}<H<1$.
Then  \begin{eqnarray*}\label{cv p.s. of theta
hat}\widehat{\theta}_{n}\longrightarrow \theta\end{eqnarray*} almost
surely as $n\rightarrow \infty$.
\end{theorem}

\begin{proof}
Using the decomposition (\ref{decomp. of LSE fOU PF}) we can write
$\widehat{\theta}_{n}=\theta+
\left(nQ_{n}^{-1}\right)\left(\frac{1}{n}R_{n}\right)$. By combining
this with Propositions \ref{cv a.s. frac{1}{n}R_{n}} and \ref{cv
a.s. nQ{-1}_{n}} below the result follows at once.
\end{proof}

\begin{proposition}\label{cv a.s. frac{1}{n}R_{n}}Assume that $\frac12<H<1$.
Then,   as $n\rightarrow\infty$
 \[\frac{1}{n}R_{n} \longrightarrow0 \]
almost surely.
\end{proposition}
\begin{proof}
 Since
\begin{eqnarray}\label{varphi bounded}
  \sup_{t\geq0}|\varphi_{i}(t)|\leq C<\infty,\quad i=1,\ldots,p
\end{eqnarray} we have
\begin{eqnarray*}
   E\left[\left(\int_{0}^{n}\varphi_{i}(t)dB_{t}^{H}\right)^{2}\right]
   &=&H(2H-1)\int_{0}^{n}\int_{0}^{n}\varphi_{i}(u)\varphi_{i}(v)|u-v|^{2H-2}dudv \\
   &\leq& H(2H-1)C^{2}\int_{0}^{n}\int_{0}^{n}|u-v|^{2H-2}dudv=C^{2}n^{2H}.
\end{eqnarray*}
Then
\begin{eqnarray*}
 \left \| \frac{1}{n}\int_{0}^{n}\varphi_{i}(t)dB_{t}^{H}\right \|_{L^{2}(\Omega)} &\leq&
 Cn^{H-1}.
\end{eqnarray*}
Combining this with the fact that
$\int_{0}^{n}\varphi_{i}(t)dB_{t}^{H}$ is Gaussian and  Lemma
\ref{lemma:principal lemma} in the Appendix, we obtain for every $i=1,\dots,p$
\begin{eqnarray*}
  \frac{1}{n}\int_{0}^{n}\varphi_{i}(t)dB_{t}^{H} &\rightarrow& 0
\end{eqnarray*}  almost  surely  as $n\rightarrow\infty$.\\
Let us now compute the limit for the last component of
$\frac{1}{n}R_n$. Using the link between the divergence integral and
the path-wise integral we have
\begin{eqnarray}
  \int_{0}^{n}X_{t}\delta B_{t}^{H}
  &=& \int_{0}^{n}X_{t} dB_{t}^{H}- H(2H-1)
  \int_{0}^{n}\int_{0}^{t}D_{s}X_{t}(t-s)^{2H-2}dsdt.\label{link young-skorohod}
\end{eqnarray}
By (\ref{explicit exp. of FOU PF}) and (\ref{FOU}) we can write
\begin{eqnarray*}
  \frac{1}{n}\int_{0}^{n}X_{t} dB_{t}^{H} &=&
    \frac{1}{n}\int_{0}^{n}(h(t)+Z_{t})(dZ_{t}+\alpha Z_{t}dt)\\
   &=& \frac{1}{n}\int_{0}^{n}h(t)dZ_{t}+\frac{\alpha}{n}\int_{0}^{n}h(t)Z_{t}dt
   +\frac{\alpha}{n}\int_{0}^{n}Z^{2}_{t}dt+\frac{1}{n}\int_{0}^{n}Z_{t}dZ_{t}\\
   &=& \frac{1}{n}\int_{0}^{n}h(t)dZ_{t}+\frac{\alpha}{n}\int_{0}^{n}h(t)Z_{t}dt+\frac{\alpha}{n}
   \int_{0}^{n}Z^{2}_{t}dt+\frac{1}{n}\int_{0}^{n}Z_{t} dZ_{t}\\
   &=&\frac{Z_{n}h(n)}{n}-\frac{1}{n}\int_{0}^{n}h'(t)Z_{t}dt+\frac{\alpha}{n}\int_{0}^{n}h(t)Z_{t}dt+\frac{\alpha}{n}
   \int_{0}^{n}Z^{2}_{t}dt+\frac{Z_{n}^2}{2n}.
\end{eqnarray*}
Furthermore
\begin{eqnarray}
  h(t) &=& \tilde{h}(t)-e^{-\alpha t}\tilde{h}(0)\label{decomp. h}
\end{eqnarray}
and
\begin{eqnarray}
  Z(t) &=& \tilde{Z}(t)-e^{-\alpha t}\tilde{Z}(0)\label{decomp. Z}
\end{eqnarray}
where
\begin{eqnarray}
  \tilde{h}(t) &:=& e^{-\alpha t}\sum_{i=1}^{p}\mu_{i}\int_{-\infty}^{t}e^{\alpha
  s}\varphi_{i}(s)ds \label{tilde(h)}
\end{eqnarray}
which is  periodic with period 1, and
\begin{eqnarray*}
  \tilde{Z}_{t}&:=& e^{-\alpha t}\int_{-\infty}^{t}e^{\alpha
  s}dB_{s}^{H}
\end{eqnarray*}
which    is a stationary and ergodic process (see \cite{CKM}). Then
the ergodic theorem   implies that,   almost  surely
\begin{eqnarray*}
 \lim_{n\rightarrow \infty} \frac{1}{n}\int_{0}^{n}h'(t)Z_{t}dt
 =\lim_{n\rightarrow \infty}
 \frac{Z_{n}h(n)}{n}
 =\lim_{n\rightarrow \infty}\frac{\alpha}{n}\int_{0}^{n}h(t)Z_{t}dt=\lim_{n\rightarrow
 \infty}\frac{Z_{n}^2}{2n}=0;
\end{eqnarray*}
\begin{eqnarray*}
\lim_{n\rightarrow\infty}\frac{\alpha}{n}\int_{0}^{n}Z^{2}_{t}dt=\alpha^{1-2H}H\Gamma(2H).
\end{eqnarray*}
Thus,   almost  surely
\begin{eqnarray*}
  \lim_{n\rightarrow \infty}\frac{1}{n}\int_{0}^{n}X_{t}
  dB_{t}^{H}=\alpha^{1-2H}H\Gamma(2H).
\end{eqnarray*}
Combining this with (\ref{link young-skorohod}) and
\begin{eqnarray*}
\lim_{n\rightarrow\infty}
\frac{H(2H-1)}{n}\int_{0}^{n}\int_{0}^{t}D_{s}X_{t}(t-s)^{2H-2}dsdt
 &=&\alpha^{1-2H}H\Gamma(2H)
\end{eqnarray*}
 we deduce that,  almost  surely
 \begin{eqnarray*}
  \lim_{n\rightarrow \infty}\frac{1}{n}\int_{0}^{n}X_{t}
  \delta B_{t}^{H}=0
\end{eqnarray*}
which completes the proof.
\end{proof}
\\

Let us now discuss   the asymptotic normality of the LSE
$\widehat{\theta}_{n}$ of $\theta$.
\begin{theorem}\label{asymp. norm. FOU with PF} Assume that $1/2<H<3/4$. Then
\begin{eqnarray}
  n^{1-H}(\widehat{\theta}_{n}-\theta) &\overset{law}{\longrightarrow}&
  \mathcal{N}(0,M^{\top}\Sigma M) \label{cv in law of theta FOU}
\end{eqnarray}
where the matrix $M$ is defined in Proposition \ref{cv a.s.
nQ{-1}_{n}}, and
\begin{eqnarray*}
  \Sigma&:=& \left(
                  \begin{array}{cc}
                    G & -a \\
                    -a^{\top} & b \\
                  \end{array}
                \right)
\end{eqnarray*}
with
\begin{eqnarray*}
  a^{\top} := \left(H(2H-1)\int_{0}^{1}\int_{0}^{1}\varphi_{i}(u)\tilde{h}(v)|v-u|^{2H-2}dudv
              \right)_{1\leq i\leq p};
\end{eqnarray*}
\begin{eqnarray*}
G := \left(\int_{0}^{1}\int_{0}^{1}\varphi_{i}(u)\varphi_{j}(v)dudv
              \right)_{1\leq i,j\leq p};
  \quad b:=H(2H-1)\int_{0}^{1}\int_{0}^{1}\tilde{h}(u)\tilde{h}(v)|v-u|^{2H-2}dudv.
\end{eqnarray*}
\end{theorem}
\begin{proof}From (\ref{decomp. of LSE fOU PF}) we have
\begin{eqnarray*}
n^{1-H}\left(\widehat{\theta}_{n}-\theta\right)&=&
\left(nQ_{n}^{-1}\right)\left(n^{-H}R_{n}\right).
\end{eqnarray*}
From Proposition \ref{cv a.s. nQ{-1}_{n}} we have
$nQ_{n}^{-1}\rightarrow M$ almost surely. Then, to prove (\ref{cv in
law of theta FOU}) it is sufficient to show that, as
$n\rightarrow\infty$
\begin{eqnarray*}
 n^{-H}R_{n}=\left(n^{-H}\int_{0}^{n}\varphi_{1}(t)dB^{H}_{t},\ldots,n^{-H}\int_{0}^{n}\varphi_{p}(t)dB^{H}_{t},
 -n^{-H}\int_{0}^{n}X_{t}\delta B^{H}_{t}\right) &\overset{law}{\longrightarrow}&
 \mathcal{N}(0,\Sigma).
\end{eqnarray*}
According to (\ref{explicit exp. of FOU PF})
\[n^{-H}\int_{0}^{n}X_{t}\delta
B^{H}_{t}=n^{-H}\int_{0}^{n}Z_{t}\delta
B^{H}_{t}+n^{-H}\int_{0}^{n}h(t)dB^{H}_{t}.\] Moreover, it follows
from \cite{HN}  that if $1/2<H<3/4$,
$n^{-1}E\left[\left(\int_{0}^{n}Z_{t}dB^{H}_{t}\right)^2\right]$
  converges to a positive constant as $n\rightarrow\infty$. This
  implies that, as $n\rightarrow\infty$
 \begin{eqnarray*}
E\left[\left(n^{-H}\int_{0}^{n}Z_{t}dB^{H}_{t}\right)^2\right]
\longrightarrow 0.
\end{eqnarray*}
It is also clear that for every  $  1\leq i\leq p$
\begin{eqnarray*}
  E\left[\left(\int_{0}^{n}\varphi_{i}(t)dB^{H}_{t}\right)\left(\int_{0}^{n}Z_t\delta B^{H}_{t}\right)\right]  &=&
  0.
\end{eqnarray*}
Indeed, this follows from the fact that the first integral can be viewed as an element in the first Wiener chaos and the second integral as an element in the second Wiener chaos.
Hence  it remains to check
\begin{eqnarray*}
\left(n^{-H}\int_{0}^{n}\varphi_{1}(t)dB^{H}_{t},\ldots,n^{-H}\int_{0}^{n}\varphi_{p}(t)dB^{H}_{t},
-n^{-H}\int_{0}^{n}h(t)dB^{H}_{t}\right)
\overset{law}{\longrightarrow} \mathcal{N}(0,\Sigma).
\end{eqnarray*}
  By using (\ref{decomp. h}) and the fact that the functions
$\tilde{h},\varphi_{i}, i=1,\ldots, p$ are periodic  functions with
period 1 it   is enough to prove that if $f_{k}, k=1,\ldots, q$ are
periodic real valued functions with period 1, then for every $H>1/2$
we have, as $n\rightarrow\infty$
\begin{eqnarray*}
 \left(n^{-H}\int_{0}^{n}f_{k}(t)dB_{t}^{H}\right)_{1\leq k\leq q}&\overset{law}{\longrightarrow}&
 \mathcal{N}\left(0,\left(\int_{0}^{1}\int_{0}^{1}f_{k}(x)f_{l}(y)dxdy\right)_{1\leq k,l\leq
 q}\right).
\end{eqnarray*}
Because the left-hand side is a Gaussian
vector  it is sufficient to check the convergence of its covariance
matrix. Since the functions $f_{k}, k=1,\ldots, q$ are periodic with
period 1, we have for every $1\leq k,l\leq q,\ i\geq1$
\begin{eqnarray*}
  E\left[ \int_{i-1}^{i}f_{k}(t)dB^{H}_{t}\int_{i-1}^{i}f_{l}(t)dB^{H}_{t}\right]
  =H(2H-1)\int_{0}^{1}\int_{0}^{1}f_{k}(t)f_{l}(s)|t-s|^{2H-2}dsdt.
\end{eqnarray*}
Hence, for every $1\leq k,l\leq q$
\begin{eqnarray*}
E\left(\int_{0}^{n}f_{k}(s)dB^{H}_{s}\int_{0}^{n}f_{l}(t)dB^{H}_{t}\right)
&=&\sum_{i,j=1}^{n}E\left(\int_{i-1}^{i}f_{k}(s)dB^{H}_{s}\int_{j-1}^{j}f_{l}(t)dB^{H}_{t}\right)\\
&=&H(2H-1)\left[n\int_{0}^{1}\int_{0}^{1}f_{k}(x)f_{l}(y)|y-x|^{2H-2}dxdy
\right.\\&&\left.+\sum_{i\neq
j=1}^{n}\int_{0}^{1}\int_{0}^{1}f_{k}(x)f_{l}(y)|j-i+y-x|^{2H-2}dxdy\right].
\end{eqnarray*}
Furthermore, for every $x,y\in[0,1]$
\begin{eqnarray*}
\sum_{i<j=1}^{n}|j-i+y-x|^{2H-2}&=&\sum_{i<j=1}^{n}\left(j-i+y-x\right)^{2H-2}\\
 &=& \sum_{i=1}^{n-1}\sum_{j=i+1}^{n}\left(j-i+y-x\right)^{2H-2} \\
&=& \sum_{m=1}^{n-1}(n-m)\left(m+y-x\right)^{2H-2}.
\end{eqnarray*}
We have
$\left(m+y-x\right)^{2H-2}\underset{\infty}{\sim}m^{2H-2}$, and
$m\left(m+y-x\right)^{2H-2}\underset{\infty}{\sim}m^{2H-1}$. Hence, since $H>\frac12$, we get
\[n\sum_{m=1}^{n-1}
\left(m+y-x\right)^{2H-2}\underset{\infty}{\sim}\frac{n^{2H}}{2H-1};
\quad
\sum_{m=1}^{n-1}m\left(m+y-x\right)^{2H-2}\underset{\infty}{\sim}\frac{n^{2H}}{2H}.\]
This implies that, as $n\rightarrow\infty$
\begin{eqnarray*}
n^{-2H}\sum_{i<j=1}^{n}|j-i+y-x|^{2H-2}\longrightarrow\frac{1}{2H-1}-\frac{1}{2H}=\frac{1}{2H(2H-1)}.
\end{eqnarray*}
Similarly,
\begin{eqnarray*}
n^{-2H}\sum_{j>i=1}^{n}|j-i+y-x|^{2H-2}\longrightarrow\frac{1}{2H(2H-1)}.
\end{eqnarray*}
As a consequence, as $n\rightarrow\infty$
\begin{eqnarray*}
n^{-2H}E\left(\int_{0}^{n}f_{k}(s)dB^{H}_{s}\int_{0}^{n}f_{l}(t)dB^{H}_{t}\right)
\longrightarrow\int_{0}^{1}\int_{0}^{1}f_{k}(x)f_{l}(y)dxdy
\end{eqnarray*}
which implies the desired result.
\end{proof}
\begin{remark}
It seems challenging to obtain the limiting behaviour of our estimator in the case
$H\geq \frac34$, and the same phenomena is present even in the case of fOU-process
 without periodicities (see, e.g. \cite{HN, sot-vii}). On the other hand, this is in analogue
  with the quadratic variations of the fractional Brownian motion in which case the limit distribution
   is not normal in the case $H>\frac34$.
\end{remark}
\section{LSE for fOU  of second kind  with periodic mean}
From (\ref{eq: SDE second kind fOU}) and (\ref{LSE FOUSK PF}) we can
write
\begin{eqnarray}
\label{decomp. LSE FOUSK PF} \widetilde{\theta}_{n}=\theta+
\widetilde{Q}_{n}^{-1}\widetilde{R}_{n},
\end{eqnarray}
where
\begin{eqnarray*}
\label{eq:vector R second kind} \widetilde{R}_{n} &:=&
\left(\int_{0}^{n}\varphi_{1}(t)dY_{t}^{(1)},\ldots.,\int_{0}^{n}\varphi_{p}(t)dY_{t}^{(1)},-\int_{0}^{n}X^{(1)}_{t}\delta
Y_{t}^{(1)}\right)^{\top},
\end{eqnarray*}
and
\begin{eqnarray}
\label{eq:matrice Q FOUSK}
  \widetilde{Q}_{n}^{-1} &=&\frac{1}{n}\left(
          \begin{array}{cc}
            I_{p}+ \eta_{n} \Gamma_{n} \Gamma_{n}^{\top} & -\eta_{n} \Gamma_{n} \\
            -\eta_{n} \Gamma_{n}^{\top} & \eta_{n} \\
          \end{array}
        \right) \nonumber
\end{eqnarray}
with
\begin{eqnarray*}
  \Gamma_{n} &=& (\Gamma_{n,1},\ldots,\Gamma_{n,p})^{\top}:=\left(\frac{1}{n}\int_{0}^{n}\varphi_{1}
  (t)X^{(1)}_{t}dt,\ldots,\frac{1}{n}\int_{0}^{n}\varphi_{p}(t)X^{(1)}_{t}dt\right)^{\top}\\
\end{eqnarray*}
and
\begin{eqnarray*}
   \eta_{n}&:=&\left(\frac{1}{n}\int_{0}^{n}(X^{(1)}_{t})^{2}dt-\sum_{i=1}^{p}\Gamma^{2}_{n,i}\right)^{-1}.
\end{eqnarray*}
\begin{theorem}\label{consitency FOUSK with PF}Assume that $1/2<H<1$. Then
\begin{eqnarray}\label{consistensy FOUSK}\widetilde{\theta}_{n}\longrightarrow\theta\end{eqnarray}
  almost surely   as $n\rightarrow \infty$.
\end{theorem}
\begin{proof}
By (\ref{decomp. LSE FOUSK PF}) we have
$\widetilde{\theta}_{n}-\theta=\left(n\widetilde{Q}_{n}^{-1}\right)\left(\frac{1}{n}\widetilde{R}_{n}\right)$.
Thus the convergence (\ref{consistensy FOUSK}) is a direct
consequence of Propositions \ref{pro:convergence vector R second
kind} and \ref{pro:matrice overlineM second kind} below.
\end{proof}
\begin{proposition}\label{pro:convergence vector R second kind} Assume that $\frac{1}{2}<H<1$.
 Then the sequence $\frac{1}{n}\widetilde{R}_{n}$ almost surely to $0$ as $n\rightarrow\infty$.
\end{proposition}
\begin{proof}Applying Proposition \ref{isometry with respect Y^1} and (\ref{varphi bounded}), we have for every $i=1,\ldots,p$
\begin{eqnarray*}
E\left[\left(\int_{0}^{n}\varphi_{i}(t)dY_{t}^{(1)}\right)^{2}\right]
   &\leq& C^2\int_{0}^{n}\int_{0}^{n}r_{H}(u,v)dudv\nonumber\\
   &=& C^2
   E[(Y_{n}^{(1)})^{2}].
\end{eqnarray*}
Thanks to \cite{CKM,EET},
\begin{eqnarray}
E\left[(Y_{n}^{(1)})^{2}\right] &=& O(n)\quad \mbox{ as }
n\rightarrow \infty. \label{variance Y^1}
\end{eqnarray}
 Thus
\begin{eqnarray*}
 \left\| \frac{1}{n}\int_{0}^{n}\varphi_{i}(t)dY_{t}^{(1)}\right
 \|_{L^{2}(\Omega)}&=& O(n^{-1/2}).
\end{eqnarray*}
Hence we can apply Lemma \ref{lemma:principal lemma} to obtain, as
$n\rightarrow \infty$
\begin{eqnarray}
  \frac{1}{n}\int_{0}^{n}\varphi_{i}(t)dY_{t}^{(1)} &\rightarrow&
  0\label{cv int varphi FOUSK}
\end{eqnarray}almost  surely for   every $i=1,\ldots,p$.\\
In order to compute the variance of the last component of
$\frac{1}{n}\widetilde{R}_{n}$, observe that we may write  the
solution of
 (\ref{eq: SDE second kind fOU}) as follows
\begin{eqnarray}
  X^{(1)}_{t} &=& h(t)+Z^{(1)}_{t}\label{decomp. of X(1)}
\end{eqnarray}
where the function $h$ is defined in (\ref{exp. of h and Z}), and
\begin{eqnarray}
  Z^{(1)}_{t} &:=& e^{-\alpha t}\int_{0}^{t}e^{\alpha
  s}dY_{s}^{(1)}.\label{Z(1)}
\end{eqnarray}
Hence
\begin{eqnarray}
  \int_{0}^{n}X^{(1)}_{t}\delta Y_{t}^{(1)} &=& \int_{0}^{n}h(t)dY_{t}^{(1)}+\int_{0}^{n}Z^{(1)}_{t}\delta Y_{t}^{(1)}.\nonumber
\end{eqnarray}
As in (\ref{cv int varphi FOUSK}),
\begin{eqnarray*}
  \frac{1}{n}\int_{0}^{n}h(t)dY_{t}^{(1)} &\rightarrow& 0
\end{eqnarray*}   almost  surely as $n\rightarrow\infty$.\\
From \cite[Lemma 3.1]{AM} we have
\begin{eqnarray}\label{identity law Z and G}
  \int_{0}^{n}Z^{(1)}_{t}\delta Y_{t}^{(1)} &\overset{law}{=}&\int_{0}^{n}\tilde{Z}^{(1)}_{t}\delta \tilde{G}_{t}
\end{eqnarray}
where the processes $\tilde{Z}^{(1)}$ and $\tilde{G}$ are well defined in (\ref{tildeZ(1)})
and (\ref{tildeG}) respectively.\\
Moreover, it  follows from \cite[Theorem 3.2]{AM} that there exists
a positive constant $\lambda(\theta,H)>0$ such that, as
$n\rightarrow\infty$
\begin{eqnarray}
 E\left[\left(\frac{1}{\sqrt{n}}\int_{0}^{n}\tilde{Z}^{(1)}_{t}\delta\tilde{G}_{t}\right)^2\right]  \longrightarrow
 \lambda(\theta,H).\label{cv variance FOUSK}
\end{eqnarray}
Combining (\ref{identity law Z and G}), (\ref{cv variance FOUSK}),
(\ref{hypercontractivity}) and Lemma \ref{lemma:principal lemma} we
  conclude that,  almost  surely
\begin{eqnarray*}
  \lim_{n\rightarrow\infty}\frac{1}{n}\int_{0}^{n}Z^{(1)}_{t}\delta Y_{t}^{(1)}
  &=&0
\end{eqnarray*}
which finishes the proof.
\end{proof}
\begin{proposition}\label{pro:matrice overlineM second kind} Assume that $\frac{1}{2}<H<1$.
 Then, as $n\rightarrow \infty$
\begin{eqnarray}
n\widetilde{Q}^{-1}_{n}\longrightarrow \overline{M} &:=&\left(
          \begin{array}{cc}
            I_{p}+ \eta \Gamma \Gamma^{t} & -\eta \Gamma \\
            -\eta \Gamma^{t} & \eta \\
          \end{array}
        \right),\label{cv of nQ^-1 FOUSK}
\end{eqnarray}almost surely,
where
\begin{eqnarray*}
  \Gamma &=& (\Gamma_{1},\ldots,\Gamma_{p})^{t}
  :=\left(\int_{0}^{1}\varphi_{1}(t)\tilde{h}(t)dt,\ldots,\int_{0}^{1}\varphi_{p}(t)\tilde{h}(t)dt\right)^{t};
  \\ \eta&:=&\left(\int_{0}^{1}\tilde{h}^{2}(t)dt
  +\frac{(2H-1)H^{2H}}{\alpha}\beta\left((\alpha-1)H+1,2H-1\right)-\sum_{i=1}^{p}\Lambda^{2}_{i}\right)^{-1},
\end{eqnarray*}
with $\tilde{h}$ is given in (\ref{tilde(h)}).
\end{proposition}
\begin{proof}
Define
$$\bar{X}^{(1)}_{t}=\tilde{h}(t)+\bar{Z_{t}}^{(1)}$$
where
\begin{eqnarray*}
\label{eq:relationship between Z_bar and Z}
  \bar{Z_{t}}^{(1)}&=&e^{-\alpha t}\int_{-\infty}^{t}e^{(\alpha-1)s}dB_{a_{s}}=Z_{t}^{(1)}+e^{-\alpha
  t}\bar{Z_{0}}^{(1)}.
\end{eqnarray*}
Since the process $\bar{Z}^{(1)}$ is   ergodic (see \cite{KS}),  as
$n\rightarrow\infty$
\begin{eqnarray}
  \frac{1}{n}\int_{0}^{n}(\bar {Z}_{t}^{(1)})^{2}dt &\rightarrow& E(\bar {Z}_{0}^{(1)})^{2} \nonumber
\end{eqnarray}
 almost surely. Hence
\begin{eqnarray}
\label{eq:ergodicity of Z1}
   \frac{1}{n}\int_{0}^{n}(Z^{(1)}_{t})^{2}dt &\rightarrow& E(\bar {Z_{0}}^{(1)})^{2}.
\end{eqnarray} almost surely as $n\rightarrow\infty$. Moreover
\begin{eqnarray*}
\label{eq:variance Z bar zero}
  E(\bar{Z}_{0}^{(1)})^{2} &=& H^{-2(\alpha-1)H}E\left(\int_{0}^{a_{0}}s^{(\alpha-1)H}dB_{s}\right)^{2} \nonumber \\
   &=& H^{-2(\alpha-1)H}H(2H-1)\int_{0}^{a_{0}}\int_{0}^{a_{0}}s^{(\alpha-1)H}t^{(\alpha-1)H}|s-t|^{2H-2}dsdt \nonumber \\
   &=& \frac{(2H-1)H^{2H}}{\alpha}\beta((\alpha-1)H+1,2H-1).
\end{eqnarray*}
Thus,
\begin{eqnarray}
  \lim_{n\rightarrow \infty}\frac{1}{n}\int_{0}^{n}(X^{(1)}_{t})^{2}dt &=& \lim_{n\rightarrow \infty}\frac{1}{n}\int_{0}^{n}(\bar{X}^{(1)}_{t})^{2}dt \nonumber\\
   &=& \int_{0}^{1}\tilde{h}(t)^{2}dt+ \frac{(2H-1)H^{2H}}{\alpha}\beta((\alpha-1)H+1,2H-1),\nonumber
\end{eqnarray}
On the other hand, we also have
\begin{eqnarray*}
  \lim_{n\rightarrow \infty} \Gamma_{n,i}&=&\lim_{n\rightarrow \infty}\frac{1}{n} \int_{0}^{n}X^{(1)}_{t}\varphi_{i}(t)dt=\lim_{n\rightarrow \infty}\frac{1}{n} \int_{0}^{n}\bar{X}^{(1)}_{t}\varphi_{i}(t)dt\nonumber\\
  &=&E\left[\int_{0}^{1}\bar{X}^{(1)}_{t}\varphi_{i}(t)dt\right]
  \\
  &=&\int_{0}^{1}\tilde{h}(t)\varphi_{i}(t)dt+\int_{0}^{1}E\left[\int_{-\infty}^{t}\varphi_{i}(t)e^{-\alpha(t-s)}dY^{(1)}_{s}\right]dt,\nonumber
   \\
  &=&\int_{0}^{1}\tilde{h}(t)\varphi_{i}(t)dt
\end{eqnarray*}
Furthermore,
\begin{eqnarray*}
  \lim_{n\rightarrow\infty}\eta_{n} &=& \lim_{n\rightarrow\infty} \left(\frac{1}{n}\int_{0}^{n}(X^{(1)}_{t})^{2}dt-\sum_{i=1}^{p}\Gamma^{2}_{n,i}\right)^{-1} \nonumber \\
   &=&\left(\int_{0}^{1}\tilde{h}(t)^{2}dt+ \frac{(2H-1)H^{2H}}{\alpha}B((\alpha-1)H+1,2H-1)-\sum_{i=1}^{p}\Gamma^{2}_{i}\right)^{-1} \nonumber
   \\
   &=&\eta.
\end{eqnarray*}
Using the fact that the functions $\varphi_{i};\ i=1,\ldots,p$ are
orthonormal in $L^{2}[0,1]$ and the Bessel inequality we get
\begin{eqnarray*}
  \sum_{i=1}^{p}\Gamma^{2}_{i} &=& \sum_{i=1}^{p}\left(\int_{0}^{1}\varphi_{i}(t)\tilde{h}(t)dt\right)^{2}\leq \int_{0}^{1}\tilde{h}^{2}(t)dt.
\end{eqnarray*}
This implies that the limit $\eta$  is well defined and finite, which completes the proof.
\end{proof}
\\

Let us now study the asymptotic normality of the LSE
$\widetilde{\theta}_{n}$ of $\theta$.
\begin{theorem}\label{asymp. norm. FOUSK with PF} Assume that $H\in(1/2,1)$.
Then, as $n\rightarrow\infty$
\begin{eqnarray}
  \sqrt{n}(\widetilde{\theta}_{n}-\theta) &\overset{law}{\longrightarrow}&
  \mathcal{N}(0,\overline{M}^{\top}\overline{\Sigma}\ \overline{M}) \label{cv in law of theta FOUSK}
\end{eqnarray}
where the matrix $\overline{M}$ is defined in Proposition
\ref{pro:matrice overlineM second kind}, and
\begin{eqnarray*}
  \overline{\Sigma}&:=& \left(
                  \begin{array}{cc}
                    \overline{G} & -\overline{a} \\
                    -\overline{a}^{\top} & \overline{b} \\
                  \end{array}
                \right)
\end{eqnarray*}
with
\begin{eqnarray*}
  \overline{a}^{\top} &:=& \left(\int_{0}^{1}\int_{0}^{1}\varphi_{i}(x)\tilde{h}(y)\sum_{m\in
\mathbb{Z}}^{\infty}r_H(x,y+m)dxdy\right)_{1\leq i\leq p};
\\
\overline{G }&:=&
\left(\int_{0}^{1}\int_{0}^{1}\varphi_{i}(x)\varphi_{j}(y)\sum_{m\in
\mathbb{Z}}^{\infty}r_H(x,y+m)dxdy\right)_{1\leq i,j\leq p};
  \\ \overline{b}&:=&\int_{0}^{1}\int_{0}^{1}\tilde{h}(x)\tilde{h}(y)\sum_{m\in
\mathbb{Z}}^{\infty}r_H(x,y+m)dxdy+\sigma^2.
\end{eqnarray*}
\end{theorem}
\begin{proof}We can write
\begin{eqnarray*}
\sqrt{n}\left(\widetilde{\theta}_{n}-\theta\right)&=&
\left(n\widetilde{Q}_{n}^{-1}\right)\left(\frac{1}{\sqrt{n}}\widetilde{R}_{n}\right)
\end{eqnarray*}
From (\ref{cv of nQ^-1 FOUSK}) we have
$n\widetilde{Q}_{n}^{-1}\rightarrow \overline{M}$ almost surely as
$n\rightarrow\infty$. Then, to prove (\ref{cv in law of theta
FOUSK}) it is sufficient to show that, as $n\rightarrow\infty$
\begin{eqnarray*}
 \frac{1}{\sqrt{n}}\widetilde{R}_{n}
 &\overset{law}{\longrightarrow}&
  \mathcal{N}(0,\overline{\Sigma}).
\end{eqnarray*}
Hence by using the main results of \cite{PT} and \cite{peccati} together with the fact that
$\frac{1}{\sqrt{n}}\widetilde{R}_{n}$ is a vector of multiple
integrals it is sufficient to check the convergence of the
covariance matrix  of $\frac{1}{\sqrt{n}}\widetilde{R}_{n}$ as
$n\rightarrow\infty$.\\
 Since ${X}^{(1)}$ admits the decomposition (\ref{decomp. of X(1)}),
 and
 \begin{eqnarray*}
\frac{1}{\sqrt{n}}\int_{0}^{n}{Z_{t}}^{(1)}\delta Y_{t}^{(1)}
\overset{law}{\longrightarrow}\mathcal{N}(0,\sigma^2)\quad
(\mbox{see } \cite{AM}),
\end{eqnarray*}
and   for every  $1\leq i\leq p$
\begin{eqnarray*}
  E\left(\int_{0}^{n}\varphi_{i}(t)dY^{(1)}_{t}
 \int_{0}^{n}{Z_{t}}^{(1)}\delta Y_{t}^{(1)}\right)
   = E\left(\int_{0}^{n}h(t)dY^{(1)}_{t}
 \int_{0}^{n}{Z_{t}}^{(1)}\delta Y_{t}^{(1)}\right) =
  0,
\end{eqnarray*}
it remains to prove that, if $f$ and $g$ are two  periodic functions
with period 1 then, as $n\rightarrow\infty$,
\begin{eqnarray}
\frac{1}{n}E\left(\int_{0}^{n}f(s)dY^{(1)}_{s}\int_{0}^{n}g(t)dY^{(1)}_{t}\right)
&\longrightarrow&  \int_{0}^{1}\int_{0}^{1}f(x)g(y)\sum_{m\in
\mathbb{Z}}r_H(x,y+m)dxdy.\label{cv inner. scal.. FOUSK}
\end{eqnarray}
 Thanks to
(\ref{inn.scal.FOUSK}),
\begin{eqnarray*}
E\left(\int_{0}^{n}f(s)dY^{(1)}_{s}\int_{0}^{n}g(t)dY^{(1)}_{t}\right)
&=&\sum_{i,j=1}^{n}E\left(\int_{i-1}^{i}f(s)dY^{(1)}_{s}\int_{j-1}^{j}g(t)dY^{(1)}_{t}\right)\\
&=& n\int_{0}^{1}\int_{0}^{1}f(x)g(y)r_H(x,y)dxdy  \\&& +\sum_{i\neq
j=1}^{n}\int_{0}^{1}\int_{0}^{1}f(x)g(y)r_H(x+i,y+j)dxdy.
\end{eqnarray*}
We also have for every $x,y\in[0,1]$
\begin{eqnarray*}
\frac{1}{n}\sum_{i<j=1}^{n}r_H(x+i,y+j)&=&\frac{1}{n}\sum_{i<j=1}^{n}r_H(x,y+j-i)\\
&=& \frac{1}{n}\sum_{m=1}^{n-1}(n-m)r_H(x,y+m)\\
&=&\sum_{m=1}^{n-1}r_H(x,y+m)-\sum_{m=1}^{n-1}mr_H(x,y+m).
\end{eqnarray*}
Since $r_H(x,y+m)\sim H(2H-1)H^{2(H-1)}e^{-(\frac{1}{H}-1)(m+y-x)}$
as $m\rightarrow\infty$, we deduce that for every fixed
$x,y\in[0,1]$
\begin{eqnarray*}\sum_{m=1}^{\infty}r_H(x,y+m)<\infty,
\end{eqnarray*}
and as $n\rightarrow\infty$,
\begin{eqnarray*}\frac{1}{n}\sum_{m=1}^{\infty}mr_H(x,y+m)\longrightarrow0.
\end{eqnarray*}
Combining these convergences with the fact that $r_H$ is symmetric
we conclude (\ref{cv inner. scal.. FOUSK}), which completes the
proof.
\end{proof}

\section{Appendix}
In this section, we briefly recall some basic elements of Malliavin
calculus with respect to fBm which are helpful for some of the
arguments we use. For more details we refer to  \cite{AMN, NP, D.
Nualart}. We also give here some of the technical results used in
various proofs of this paper.
\\
Let $B^H = \left\{B^H_t, t \geq 0\right\}$ be a  fBm  with Hurst
parameter $H \in ( 0 , 1)$ that is     a centered Gaussian process
with the covariance function
\[R_{H}(t,s)=\frac12\left(t^{2H}+s^{2H}-|t-s|^{2H}\right).\] It is well-known that the
covariance function $R_{H}$ can be represented as
\begin{eqnarray*}
  R_{H}(t,s) &=& \int_{0}^{t \wedge s}K_{H}(t,u)K_{H}(s,u)du
\end{eqnarray*}
where, in the case when $H>\frac{1}{2}$, the kernel $K_{H}$ has a
explicit expression given by
\begin{eqnarray*}
  K_{H}(t,s) &=&c_{H}s^{\frac{1}{2}-H}\int_{s}^{t}(u-s)^{H-\frac{3}{2}}u^{H-\frac{1}{2}}du,\quad
  s<t,
\end{eqnarray*}
where
$c_{H}=(H-\frac{1}{2})\left(\frac{2H\Gamma(\frac{3}{2}-H)}{\Gamma(H+\frac{1}{2})\Gamma(2-2H)}\right)^{\frac{1}{2}}$.
\\
We denote by $\mathcal E$ the set of step $\R-$valued functions on
$\oT$. Let $\mathcal H$ be the Hilbert space defined as the closure
of $\mathcal E$ with respect to the scalar product
$$
\left \langle\1_{[0,t]}, \1_{[0,s]}\right \rangle_{\mathcal
H}=R_H(t,s).
$$
We denote by $|\cdot|_{\cal H}$ the associated norm. The mapping
$\displaystyle \1_{[0,t]} \mapsto B_{t}$ can be extended to an
isometry between $\mathcal H$ and the Gaussian space associated with
$B$. We denote this isometry by
\begin{equation}\label{wienerfrac}
\ffi\mapsto B(\ffi)=\int_0^T \ffi(s)dB_s.
\end{equation}
When $H  \in (\frac{1}{2}, 1)$, it is well known that the elements
of ${\cal{H}}$ may not be functions but distributions of negative
order. It will be more convenient to work with a subspace of $\cal
H$ which contains only functions. Such a space is the set $\vert
{\cal{H}}\vert$ of all measurable functions $\ffi$ on $[0,T]$ such
that
\begin{equation*}
|\ffi|^2_{|\cal H|}:=H (2H -1)\int _{0}^{T} \int _{0}^{T} \vert
\ffi(u) \vert \vert \ffi(v)\vert \vert u-v\vert ^{2H -2} dudv
<\infty.
\end{equation*}
If $\ffi,\psi\in|\cal H|$ then
\begin{equation}\label{iso}
E\big[B(\ffi)B(\psi)\big]=H (2H -1)\int _{0}^{T} \int _{0}^{T}
\ffi(u) \psi(v) \vert u-v\vert ^{2H -2} dudv.
\end{equation}
We know that $(|\cal H|, \langle \cdot,\cdot\rangle_{|\cal H|})$ is
a Banach space, but that $(|\cal H|, \langle\cdot,\cdot\rangle_{\cal
H})$ is not complete. We have the dense inclusions $ L^{2}([0,T])
\subset L^{\frac{1}{H}}([0,T]) \subset \vert {\cal{H}}\vert \subset
{\cal{H}}. $
 Let us introduce the linear operator $K_{H}^{\ast}$ between $\mathcal{E}$ and $L^{2}[0,T]$ defined by
\begin{eqnarray*}
  (K_{H}^{\ast}\varphi)(s) &=& \int_{s}^{T}\varphi(t)\frac{\partial K_{H}}{\partial t}(t,s)dt. \nonumber
\end{eqnarray*}
The operator $K_{H}^{\ast}$ is an isometry that can be extended to
$\mathcal{H}$. Moreover, the process $W = \left\{W_t,
t\in[0,T]\right\}$  given by
\begin{eqnarray}\label{BM}
  W_{t} &:=& B^H((K^{\ast}_{H})^{-1}(1_{[0,t]}))
\end{eqnarray}
is a Brownian motion. In addition, the processes $B^H$ and $W$ are
related through the integral representation
\begin{eqnarray*}\label{integ. repres.}
  B^H_{t} &=& \int_{0}^{t} K_{H}(t,s)dW_{s}.
\end{eqnarray*}
Let $\cal S$ be the set of all smooth cylindrical random variables,
which can be expressed as $F = f(B^H(\phi_1), \ldots, B^H(\phi_n))$
where $n\geq 1$, $f : \R^n \rightarrow \R$ is a
$\mathcal{C}^\infty$-function such that $f$ and all its derivatives
have at most polynomial growth, and $\phi_i\in\cal H$,
$i=1,\ldots,n$. The Malliavin derivative of $F$ with respect to
$B^H$ is the element of $L^2(\Omega, \HH)$ defined by
\[
D_sF\; =\; \sum_{i =1}^n \frac{\partial f}{\partial
x_i}(B^H(\phi_1), \ldots, B^H(\phi_n)) \phi_i(s), \quad s \in [0,T].
\]
In particular $D_s B^H_t = \Un_{[0,t]}(s)$. As usual, $\sk^{1,2}$
denotes the closure of the set of smooth random variables with
respect to the norm
$$\| F\|_{1,2}^2 \; = \; E[F^2] +
E\big[|DF|_{\HH}^2\big].$$
 Moreover,  for any $F\in
\mathbb{D}_{W}^{1,2}$,
\begin{eqnarray*}
  K_{H}^{\ast}DF &=& D^{W}F,
\end{eqnarray*}
where $D^{W}$ denotes the Malliavin derivative operator with respect
to  $W$, and $\mathbb{D}^{1,2}_{W}$ the corresponding Hilbert space.
\\
 The Skorohod integral with respect to $B^H$ denoted by $\delta$  is the
adjoint of the derivative operator $D$. If a random variable $u \in
L^{2}(\Omega, \HH)$ belongs to the domain  of the Skorohod integral
(denoted by ${\rm dom}\delta$), that is, if it verifies
$$
|E\langle DF,u\rangle_{\cal H}|\leq c_u\,\sqrt{E[F^2]}\quad\mbox{for
any }F\in{\cal S},
$$
then  $\delta(u)$ is defined by the duality relationship
$$ E[F \delta(u)]= E\big[\langle DF, u \rangle_{\HH}\big], $$ for every $F \in \sk^{1,2}$.
In the sequel, when $t\in[0,T]$ and $u\in{\rm dom} \delta$, we shall
sometimes write $\int_0^t u_s \delta B^H_s$ instead of
$\delta(u\Un_{[0,t]})$. If $g\in \HH$, notice moreover that
$\int_0^T g_s\delta B^H_s=\delta(g)=B^H(g)$.
\\
It is known that the multiple Wiener integrals   satisfy a
hypercontractivity property, which implies that for any $F$  having
the form of a finite sum of multiple integrals, we have
\begin{equation}
\left( E\big[|F|^{p}\big]\right) ^{1/p}\leqslant c_{p,q}\left( E\big[|F|^{2}%
\big]\right) ^{1/2}\ \mbox{ for any }p\geq 2.
\label{hypercontractivity}
\end{equation}

One can also develop a Malliavin calculus for any continuous
Gaussian process G of the form (see \cite{AMN})
\begin{eqnarray*}
  G_{t} &=& \int_{0}^{t}K(t,s)dW_{s} \nonumber
\end{eqnarray*}
where $W$ is a Brownian motion and the kernel $K$  satisfying
$\sup_{t\in[0,T]}\int_{0}^{t}K(t,s)^{2}ds<\infty$. Consider the
linear operator $K^{\ast}$ from $\mathcal{E}$ to $L^{2}[0,T]$
defined by
\begin{eqnarray*}
   (K^{\ast}\varphi)(s) &=& \varphi (s)K(T,s)+\int_{s}^{T}[\varphi (t)-\varphi (s)]K(dt,s).\nonumber
\end{eqnarray*}
The Hilbert space $\mathcal{H}_G$ generated by covariance function
of the Gaussian process $G$ can be represented as
$\mathcal{H}_G=(K^{\ast})^{-1}(L^{2}[0,T])$ and
$\mathbb{D}_{G}^{1,2}(\mathcal{H}_G)=(K^{\ast})^{-1}(\mathbb{D}^{1,2}_{W}(L^{2}[0,T])).$
For any $n\geq1,$ let $\mathfrak{H}_{n}$ be the nth Wiener chaos of
$G$, i.e. the closed linear subspace of $L^{2}(\Omega)$ generated by
the random variables $\{H_{n}(G(\varphi)),\varphi \in
\mathcal{H},\|\varphi\|_{\mathcal{H}}=1\},$ and $H_{n}$ is the nth
Hermite polynomial. It is well known that the mapping
$I^{G}_{n}(\varphi^{\otimes n})=n!H_{n}(G(\varphi))$ provides a
linear isometry between the symmetric tensor product
$\mathcal{H}^{\odot n}$ and subspace $\mathfrak{H}_{n}$.
Specifically, for all $f\in\mathcal{H}_G^{\odot p},
f\in\mathcal{H}_G^{\odot q}$ and $p,q\geq 1$, one has
\begin{equation*}
E\big[I^{G}_p(f)I^{G}_q(g)\big]=\delta_{pq}q!\langle
f,g\rangle_{\mathcal{H}_G^{\otimes q}}.
\end{equation*}
 We say that the kernel K is regular if for all $s\in [0,T),$ K(.,s) has bounded variation on the interval $(s,T],$
  and
\begin{eqnarray}
  \int_{0}^{T}|K|((s,T],s)^{2}ds &<& \infty. \nonumber
\end{eqnarray}
For regular kernel K, put $K(s^{+},s):=K(T,s)-K((s,T],s)$. For any
$\varphi \in \mathcal{E},$ define the seminorm
\begin{eqnarray}
  \|\varphi\|^{2}_{Kr} &=& \int_{0}^{T} \varphi(s)^{2}K(s^{+},s)^{2}
  +\int_{0}^{T}\left(\int_{s}^{T}|\varphi(t)||K|(dt,s) \right)^{2}ds.\nonumber
\end{eqnarray}
Denote by $\mathcal{H}_{Kr}$ the completion of $\mathcal{E}$ with
respect to seminorm $\|\|_{Kr}$. \\
 The following
proposition establishes the relationship between path-wise integral
and Skorokhod integral.
\begin{proposition}\label{pro:path-wise and Skorokhod}\cite{AMN} Assume K is a regular kernel with $K(s^{+},s)=0$ and u is a process in $\mathbb{D}_{G}^{1,2}(\mathcal{H}_{K_{r}})$. Then the process u is Stratonovich integrable with respect to G and
\begin{eqnarray}
  \int_{0}^{T}u_{t}dG_{t} &=& \int_{0}^{T} u_{t}\delta G_{t}+\int_{0}^{T}\left(\int_{s}^{T}D_{s}u_{t}K(dt,s)\right)ds.\nonumber
\end{eqnarray}
\end{proposition}
\begin{proposition}\cite[Proposition 4.2]{DFW}\label{cv a.s. nQ{-1}_{n}}
Assume that $\frac{1}{2}<H<1$. Then,  as $n\rightarrow \infty$ we
obtain that $nQ^{-1}_{n}$ converges almost surely to
\begin{eqnarray}
  M &:=&\left(
          \begin{array}{cc}
            I_{p}+ \gamma \Lambda \Lambda^{t} & -\gamma \Lambda \\
            -\gamma \Lambda^{t} & \gamma \\
          \end{array}
        \right) \nonumber
\end{eqnarray}
with
\[
  \Lambda = (\Lambda_{1},\ldots,\Lambda_{p})^{t}:=\left(\int_{0}^{1}\varphi_{1}(t)\tilde{h}(t)dt,\ldots,
  \int_{0}^{1}\varphi_{p}(t)\tilde{h}(t)dt\right)^{\top}\]
  and
  \[ \gamma:=\left(\int_{0}^{1}\tilde{h}^{2}(t)dt+\alpha^{-2H}
  H\Gamma(2H)-\sum_{i=1}^{p}\Lambda^{2}_{i}\right)^{-1},
\]
where $\tilde{h}(t):=e^{-\alpha
t}\sum_{i=1}^{p}\mu_{i}\int_{-\infty}^{t}e^{\alpha
s}\varphi_{i}(s)ds.$
\end{proposition}
\begin{lemma}\cite{AM} \label{lemma:equality in law}Let $\tilde{B}_{t}=B_{t+H}-B_{H}$
 be the shifted fractional Brownian motion. Then there exists a regular Volterra-type kernel $\tilde{L}$,
  in above sense, so that for the solution of the following stochastic differential equation
\begin{eqnarray}
  d\tilde{Z}^{(1)}_{t}&=&-\alpha \tilde{Z}^{(1)}_{t}dt+d\tilde{G}_{t}, \ \
  \tilde{Z}^{(1)}_{0}=0,\label{tildeZ(1)}
\end{eqnarray}
where the Gaussian process
\begin{eqnarray}
  \tilde{G}_{t} &=&
  \int_{0}^{t}\left(K_{H}(t,s)+\tilde{L}(t,s)\right)d\tilde{W}_{s}\label{tildeG}
\end{eqnarray}
with $\tilde{W}$ is a Brownian motion  as in (\ref{BM}).\\
In addition,
$\{Z^{(1)}_{t},t\in[0,T]\}\overset{law}{=}\{\tilde{Z}^{(1)}_{t},t\in[0,T]\}$
with $Z^{(1)}$ is given in (\ref{Z(1)}).
\end{lemma}

We also need the following technical results.
\begin{lemma}\label{lemma:principal lemma}{\cite{KN}} Let $\gamma>0$ and $p_{0}\in \mathbb{N}$. Moreover let $(Z_{n})_{n\in \mathbb{N}}$ be a sequence of random variables. If for every $p\geq p_{0}$ there exists a constant $c_{p}>0$ such that for all $n\in \mathbb{N},$
\begin{eqnarray}
  (E|Z_{n}|^{p})^{1/p} &\leq& c_{p}. n^{-\gamma},\nonumber
\end{eqnarray}
then for all $\varepsilon>0$ there exists a random variable $\eta_{\varepsilon}$ such that
\begin{eqnarray}
  |Z_{n}| &\leq& \eta_{\varepsilon}.n^{-\gamma+\varepsilon} \  \ almost \ \ surely \nonumber
\end{eqnarray}
for all $n\in \mathbb{N}$. Moreover, $E|\eta_{\varepsilon}|^{p}<\infty$ for all $p\geq 1.$
\end{lemma}

\begin{proposition}\label{isometry with respect Y^1}
Assume that $\frac12<H<1$. Let $f:[0,\infty)\rightarrow\R$  be a
function of class $\mathcal{C}^1$. Then
\begin{eqnarray*}\label{link Y^1 and B}
\int_{s}^{t}f(r)dY^{(1)}_r &=&  \int_{a_s}^{a_t}f(a^{-1}_u)e^{a^{-1}_u}dB_u
\end{eqnarray*}where $a^{-1}_u=H\log(u/H)$. Moreover, for every $f,g$ of $\mathcal{C}^1$
\begin{eqnarray*}
 && E\left(\int_{s}^{t}f(r)dY^{(1)}_r\int_{u}^{v}g(r)dY^{(1)}_r\right) \nonumber\\&=&H(2H-1)\int_{a_s}^{a_t}\int_{a_u}^{a_v}
  f(a^{-1}_x)g(a^{-1}_y)e^{-a^{-1}_x}e^{-a^{-1}_y}|x-y|^{2H-2}dxdy \label{isometry Y^1}\\
   &=&  \int_{s}^{t}\int_{u}^{v}f(w)g(z)r_{H}(w,z)dwdz \label{inn.scal.FOUSK}
\end{eqnarray*}
In particular, we obtain
\begin{eqnarray*}
  E\left((Y_{t}^{(1)}-Y_{s}^{(1)})(Y_{v}^{(1)}-Y_{u}^{(1)})\right) &=&  \int_{s}^{t}\int_{u}^{v}r_{H}(w,z)dwdz
\end{eqnarray*}
where $r_{H}(x,y)$ is a symmetric kernel given by
\begin{eqnarray*}
  r_{H}(w,z) &=& H(2H-1)H^{2(H-1)} \frac{e^{-(1-H)(w-z)/H}}{\left( 1-e^{-(w-z)/H}\right)^{2(1-H)}}.
\end{eqnarray*}
\end{proposition}
\begin{proof}Combining  integration by parts and change of variable $u=a_r$ we obtain
\begin{eqnarray*}
\int_{s}^{t}f(r)dY^{(1)}_r &=& \int_{s}^{t}f(r)e^{-r}dB_{a_r}
\\&=&f(t)e^{-t}B_{a_t}-f(s)e^{-s}B_{a_s}-\int_{s}^{t}\left(f(r)e^{-r}\right)'(r)B_{a_r}dr
\\&=&f(t)e^{-t}B_{a_t}-f(s)e^{-s}B_{a_s}-\int_{a_s}^{a_t}\left(f(r)e^{-r}\right)'(a^{-1}_u)B_{a_r}(a^{-1})'(u)du
\\&=&f(t)e^{-t}B_{a_t}-f(s)e^{-s}B_{a_s}-\int_{a_s}^{a_t}\left(f(a^{-1}_u)e^{-a^{-1}_u}\right)'(u)B_{a_r}du
\\&=&  \int_{a_s}^{a_t}f(a^{-1}_u)e^{-a^{-1}_u}dB_u
\end{eqnarray*}
which proves (\ref{link Y^1 and B}). Using  (\ref{link Y^1 and B}) we get
\begin{eqnarray*}
 E\left(\int_{s}^{t}f(r)dY^{(1)}_r\int_{u}^{v}g(r)dY^{(1)}_r\right) &=& H(2H-1)\int_{a_s}^{a_t}\int_{a_u}^{a_v}
  f(a^{-1}_x)g(a^{-1}_y)e^{-a^{-1}_x}e^{-a^{-1}_y}|x-y|^{2H-2}dxdy
\end{eqnarray*}
Making now change of variable $w=a^{-1}_x$ and $z=a^{-1}_y$ we obtain (\ref{isometry Y^1}).
\end{proof}

\end{document}